\documentclass{amsart}
\usepackage{amsfonts,amssymb,amscd,amsmath,enumerate,verbatim,calc}
\usepackage{amscd,amssymb,amsopn,amsmath,amsthm,graphics,amsfonts,enumerate,verbatim,calc}
\usepackage[dvips]{graphicx}
\usepackage[colorlinks=true,linkcolor=red,citecolor=blue]{hyperref}
\input xy
\xyoption{all}
\calclayout

\newcommand{\rt}{\rightarrow}
\newcommand{\lrt}{\longrightarrow}

\newcommand{\st}{\stackrel}

\newcommand{\al}{\alpha}
\newcommand{\la}{\lambda}
\newcommand{\La}{\Lambda}
\newcommand{\Ga}{\Gamma}
\newcommand{\lan}{\langle}
\newcommand{\ran}{\rangle}

\newcommand{\D}{\mathbb{D} }

\newcommand{\CQ}{\mathcal{Q} }

\newcommand{\CT}{\mathcal{T} }
\newcommand{\CX}{\mathcal{X} }

\newcommand{\Mod}{{\rm{Mod}}}
\newcommand{\mmod}{{{\rm{mod}}}}
\newcommand{\Inj}{{\rm{Inj}}}
\newcommand{\Prj}{{\rm{Proj}}}

\newcommand{\QR}{{{\rm Rep}{(\CQ, R)}}}
\newcommand{\QLA}{{{\rm Rep}{(\CQ, \La)}}}

\newcommand{\add}{{\rm{add}}}

\newcommand{\pd}{{\rm{pd}}}

\newcommand{\derdim}{{\rm{der.dim}}}
\newcommand{\repdim}{{\rm{rep.dim}}}
\newcommand{\gldim}{{\rm{gl.dim}}}

\newcommand{\Coker}{{\rm{Coker}}}

\newcommand{\Hom}{{\rm{Hom}}}

\newcommand{\End}{{\rm{End}}}

\newtheorem{theorem}{Theorem}[section]
\newtheorem{corollary}[theorem]{Corollary}
\newtheorem{lemma}[theorem]{Lemma}

\theoremstyle{definition}

\newtheorem{remark}[theorem]{Remark}
\newtheorem{s}[theorem]{}

\theoremstyle{plain}

\theoremstyle{definition}

\numberwithin{equation}{section}

\begin{document}

\title[On the Dimensions of Path Algebras]{On the Dimensions of Path Algebras}

\author[Asadollahi and Hafezi]{Javad Asadollahi and Rasool Hafezi}

\address{Department of Mathematics, University of Isfahan, P.O.Box: 81746-73441, Isfahan, Iran and School of Mathematics, Institute for Research in Fundamental Science (IPM), P.O.Box: 19395-5746, Tehran, Iran } \email{asadollahi@ipm.ir}
\email{r.hafezi@sci.ui.ac.ir}



\thanks{The research of authors was in part supported by a grant from IPM (No. 90130216).}

\begin{abstract}
In this paper we study the representation dimension as well as the derived dimension of the path algebra of an artin algebra over a finite and acyclic quiver.
\end{abstract}

\maketitle

\section{Introduction}
The notion of representation dimension of an artin algebra $\La$, denoted $\repdim \La$, has been introduced by Auslander in his Queen Mary Notes in 1971. His hope was to provide ``a reasonable way of measuring how far the category of finitely generated modules over an artin algebra is from being of finite representation type''. Recall that an artin algebra $\La$ is called of finite representation type if, up to isomorphisms, there are only finitely many finitely generated indecomposable $\La$-modules. Such algebras are called representation finite. The importance of this class of artin algebras for the whole representation theory of artin algebras is well understood, see e.g. \cite{Au}, \cite{Ga} and \cite{Ta}.

Auslander proved that $\La$ is of finite representation type if and only if its representation dimension is at most two. But it turned out to be hard to compute the actual value of the representation dimension of a given algebra. Moreover, all the algebras that one could compute explicitly their representation dimensions, had representation dimension at most three. So a natural question arises: are there artin algebras of representation dimension greater than three. Of course such algebras are of infinite representation type. On the other hand, another question was whether this invariant is always finite. These questions has been answered positively, at almost the same time. Iyama \cite{I} in 2003 proved that it is always finite and Rouquier \cite{Rq1} in 2005 showed that any natural number $n$, except $1$, can occur as the representation dimension of some algebra. In fact, he showed that the representation dimension of the exterior algebra $\La k^n$, where $k$ is a field, is $n+1$. After it, Oppermann \cite{Opp1}, Krause and Kussin \cite{KK} and also Ringel \cite{R} constructed examples of algebras of arbitrary representation dimension.

To present his example, Rouquier introduced the notion of dimension of a triangulated category \cite{Rq2}; see also \cite{BV}. This dimension gives a new invariant for algebras and algebraic varieties under derived equivalences. Roughly speaking, it measures the minimum steps one requires to build the whole triangulated category $\CT$ out of a single object $M$, for details see the preliminaries section.

Let $\La$ be an artin algebra and $\mmod \La$ denotes the category of finitely generated (left) $\La$-modules. The dimension of the triangulated category $\D^b(\mmod \La)$, the bounded derived category, is called the derived dimension of $\La$, denoted $\derdim \La$. By Rouquier \cite{Rq2} and also Krause and Kussin \cite{KK} we know that $\derdim\La$ provides a lower bound for the representation dimension of $\La$. In fact, if $\La$ is an artin algebra and $M \in \mmod\La$ is a generator, then it is proved by Rouquier \cite[Proposition 7.4]{Rq2} that $\derdim\La \leq \gldim\End_{\La}M.$ Beligiannis \cite[Corollary 6.7]{B2} proved that if $\CX$ is a representation-finite resolving subcategory of $\mmod\La$, then $\derdim\La\leq \max\{2,\CX{\rm-}\dim(\mmod\La)\}.$

Algebras of derived dimension zero has been characterized in \cite[Theorem 12.20]{B1}, see also \cite{CYZ}. In fact, if $\La$ is a finite dimensional algebra over an algebraically closed field $k$, then $\derdim\La=0$ if and only if $\La$ is an iterated tilted algebra of Dynkin type.

A finite dimensional algebra $\La$ is called derived finite if up to shift and isomorphism there is only a finite number of indecomposable objects in $\D^b(\mmod\La)$. It is known that the hereditary algebras
\cite[Proposition 7.27]{Rq2}, the radical square zero algebras \cite[Proposition 7.37]{Rq2} and algebras of finite representation type \cite[Corollary]{H}, which are not derived finite are of derived dimension one, see
\cite[Remark 2]{H}. Orlov \cite{O} proved that a smooth quasi-projective curve has derived dimension one and Oppermann \cite{Opp2} showed that an elliptic curve and also a weighted projective line of tubular type have derived dimension one. On the other hand, Rouquier \cite{Rq2} showed that the Loewy length of a ring provide an upper bound for its derived dimension. Moreover, Krause and Kussin \cite{KK} proved that the global dimension of a ring is an upper bound for its derived dimension. Furthermore, it is known \cite{Rq2} that the dimension of the bounded derived category of perfect complexes over a ring of infinite global dimension is infinite. See \cite{OS}, for more examples of triangulated categories of infinite dimension.

In this paper, we study the above mentioned dimensions, i.e. the representation dimension and the derived dimension of path algebras of finite acyclic quivers. We show that if $\CQ$ is such a quiver, which is not of type $A_n$, then
\[\repdim\La\CQ \leq \repdim\La+5,\]
where $\La$ is an artin algebra over a commutative artin ring $k$ and $\La\CQ$ is its associated path algebra. For quivers of type $A_n$ see Remark \ref{Steffen}.

As for derived dimension, we show that, when $R$ is a coherent ring, then
\[\dim\D^b(\mmod R) \leq \dim\D^b(\mmod R\CQ) \leq 2\dim\D^b(\mmod R)+1.\]
So, in view of \cite{CYZ}, in case $k$ is an algebraically closed field and $\CQ$ is not Dynkin, we get that $\derdim k\CQ=1$. This answers, partially, a question posed by Rouquier \cite[Remark 7.29]{Rq2}.
Finally, we show that similar techniques work to prove the inequalities
\[n \leq \dim\D^b(\mmod \Ga) \leq 2n+1,\]
where $\Ga= \left( \begin{array}{ccc} R & 0 \\ M & S \end{array} \right)$ is a coherent formal triangular matrix ring and \[n=\max\{\dim\D^b(\mmod R), \dim\D^b(\mmod S)\}.\]

\section{Preliminaries}
\s {\sc Quivers and their representations.}
A quiver $\CQ$ is a directed graph. It will be denoted by a quadruple $\CQ=(\CQ_0,\CQ_1,s,t)$, where $\CQ_0$ and $\CQ_1$ are respectively the sets of vertices and arrows of $\CQ$ and $s,t:\CQ_1 \rt \CQ_0$ are two maps which associate to any arrow $\al \in \CQ_1$ its source $s(\al)$ and its target $t(\al)$, respectively. We usually denote the quiver $\CQ=(\CQ_0,\CQ_1,s,t)$ briefly by $(\CQ_0,\CQ_1)$ or even simply by $\CQ$.
A vertex $v \in \CQ_0$ is called a sink if there is no arrow $\alpha$ with $s(\alpha)=i$. $v$ is called a source if there is no arrow $\alpha$ with $t(\alpha)=v$.

A quiver $\CQ$ is said to be finite if both $\CQ_0$ and $\CQ_1$ are finite sets. A path of length $l\geq 1$ with source $a$ and target $b$ (from $a$ to $b$) is a sequence of arrows $\al_1\al_2\cdots\al_l$, where $\al_i \in \CQ_1$, for all $1\leq i \leq l$, and we have $s(\al_l)=a$, $t(\al_i)=s(\al_{i-1})$ for all $1<i\leq l$ and $t(\al_1)=b$. A path of length $l\geq 1$ is called a cycle if its source and target coincide. $\CQ$ is called acyclic if it contains no cycles. Throughout the paper we assume that $\CQ$ is a finite acyclic quiver.

A quiver $\CQ$ can be considered as a category whose objects are the vertices of $\CQ$ and morphisms are all paths in $\CQ$. Assume that $R$ is a ring. A covariant functor from $\CQ$ to $\mmod R$, the category of finitely presented $R$-modules, is called a representation of $\CQ$ in $\mmod R$. The functor category ${(\mmod R)}^{\CQ}$, whose objects are representations of $\CQ$ and morphisms are natural transformations, is known as the category of representations of $\CQ$ in $\mmod R$ and is usually denoted by $\QR$. It is known that this category is equivalent to the category of finitely presented $R\CQ$-modules, where $R\CQ$ is the path algebra of $\CQ$.

\s {\sc Evaluation functor and its adjoints.}\label{adj}
Let $v$ be a vertex of $\CQ$. Associated to $v$, there exists a functor $e^v:\QR \rt \mmod R$, called the evaluation functor, which assigns to any representation $\CX$ of $\CQ$ its module at vertex $v$, denoted $\CX_v$. It is proved in \cite{EH} that $e^v$ possesses a right and also a left adjoint, which will be denoted by $e^v_\rho$ and $e^v_\la$, respectively. In fact, $e^v_{\rho}:\mmod R \rt \QR$ is defined as follows: let $M$ be an arbitrary module in $\mmod R$. Then $e^v_{\rho}(M)_w=\prod_{\CQ(w,v)}M$, where $\CQ(w,v)$ denotes the set of paths starting in $w$ and terminating in $v$. The maps are natural projections, that is, for an arrow $a:w_1\rt w_2$, we set $e^v_{\rho}(M)_a:\prod_{\CQ(w_1,v)}M \rt \prod_{\CQ(w_2,v)}M$. The left adjoint of $e^v$ is defined similarly: for any $R$-module $M$, one defines $e^v_{\la}(M)_w=\prod_{\CQ(v,w)}M.$ The maps are natural injections.

Let $\CQ$ be as above and $v \in \CQ_0$. Based on the properties of these adjoints, one may deduce that for any projective module $P\in \Prj(R)$, the representation $e^v_{\la}(P)$ is a projective representation of $\CQ$, that is belongs to $\Prj(\QR)$. Moreover, any projective representation can be decompose into a product of such representations. In fact the set
\[\{e^v_{\la}(P) : P \in \Prj(R) \ \ {\rm{and}} \ \ v\in V\},\]
is a set of projective generators for the category $\QR$.

On the other hand, for any injective module $E \in \Inj(R)$, the representation $e^v_{\rho}(E)$, is an injective representation of $\CQ$, that is belongs to $\Inj(\QR)$. Moreover, any injective representation of $\CQ$, can be written as a coproduct of such representations. In fact, the set
\[\{e^v_{\rho}(E) : E \in\Inj(R) \ \ {\rm{and}} \ \ v\in V\},\]
is a set of injective cogenerators of $\QR$. This, in particular, implies that every representation $M$ of $\CQ$ can be embedded in a direct sum of elements of this set. The proof of these facts can be found in \cite{EE} and \cite{EER}.

The evaluation functor $e^v:\QR \rt \mmod R$ naturally extends to a triangulated functor $k^v:\D^b(\QR) \lrt \D^b(\mmod R)$ between corresponding bounded derived categories, see \cite{AEHS}. Moreover, one can sees easily that this extended functor also possesses left and right adjoints, which we will denote by $k^v_\la$ and $k^v_\rho$, respectively. In fact, these adjoints are nothing but the natural extensions of the functors $e^v_\la$ and $e^v_\rho$, respectively.

\s {\sc An exact triangle in $\D^b(\QR).$}\label{ExactTri}
Let $\CX$ be a representation of a quiver $\CQ$. It follows from the proof of \cite[Corollary 28.3]{M} that there exists an exact sequence
\[0 \rt \bigoplus_{a \in \CQ_1} e^{t(a)}_\la(\CX_{s(a)}) \rt \bigoplus_{v \in \CQ_0} e^v_\la(\CX_v) \rt \CX \rt 0,\]
in $\QR$, see also \cite[4.1]{EERI}. It is easy to check that similar sequence exists, if we replace $\CX$ with a bounded complex $X$ of representations. That is, for any bounded complex $X$ of representations of a quiver $\CQ$, we have the following short exact sequence
\[0 \rt \bigoplus_{a \in \CQ_1} k^{t(a)}_\la(X_{s(a)}) \rt \bigoplus_{v \in \CQ_0} k^v_\la(X_v) \rt X \rt 0.\]
By \cite[Chapter III, \S 3, Proposition 5]{GM}, this sequence induces the following triangle
\[\xymatrix{\bigoplus_{a \in \CQ_1} k^{t(a)}_\la(X_{s(a)}) \ar[r] & \bigoplus_{v \in \CQ_0} k^v_\la(X_v) \ar[r] & X \ar@{~>}[r] & },\]
in $\D^b(\QR).$ We shall use this triangle throughout the paper.

\s {\sc Formal triangular matrix algebras.}\label{ForTriMatAlg}
Let $R$ and $S$ be rings and $M$ be an $S$-$R$-bimodule. Let $\Ga=\left(\begin{array}{ccc} R & 0 \\
M & S \end{array} \right)$ be the formal triangular matrix ring. It is known, e.g. by \cite[Chapter III, \S 2]{ARS}, that the category $\Mod\Ga$, consisting of all (left) $\Ga$-modules, may be identified with the category $\mathfrak{C}$ consisting of all triples $(X,\,Y)_\varphi$, where $X\in\Mod R$, the category of left $R$-modules, $Y\in\Mod S$, the category of left $S$-modules and $\varphi:M\otimes_RX\lrt Y$ is an $S$-linear map.

Using this equivalence, we define two evaluation functors $e^1:\Mod\Ga\lrt\Mod R$ and $e^2:\Mod\Ga\lrt\Mod S$ by setting $e^1((X, Y)_\varphi)=X$ and $e^2((X, Y)_\varphi)=Y$, for every $\Ga$-module $(X, Y)_\varphi$. The morphisms are defined in the obvious way.

Now let $X$ be an arbitrary $R$-module. We define a functor $e^1_{\la}:\Mod R\lrt\Mod\Ga$ by setting $e^1_{\la}(X)=(X,\,M\otimes_RX)_1$, where $1:M\otimes_RX \lrt M\otimes_RX$ is the identity map. In fact,
$e^1_{\la}$ is a left adjoint of $e^1$. In another words, $(e^1_{\la},e^1)$ is an adjoint pair of functors. Similarly, we may define a functor $e^2_{\la}:\Mod S\lrt\Mod\Ga$ as follows: for any $S$-module $Y$, we define $e^2_{\la}(Y)=(0,\,Y)_0$. It is easy to check that the pair $(e^2_\lambda,\,e^2)$ is also an adjoint pair of functors.

Let $(X,Y)_{\varphi}$ be a $\Ga$-module. There exist an exact sequence
\[0 \lrt e^2_{\la}(M\otimes X) \lrt e^1_{\la}(X)\oplus e^2_{\la}(Y)_0 \lrt (X,Y)_{\varphi}\lrt 0,\]
of $\Ga$-modules. More precisely, the above sequence can be read as follows
\[0 \lrt (0,M\otimes X)_0 \st{f}{\lrt} (X,M\otimes X)_1\oplus (0,Y)_0 \st{g}{\lrt} (X,Y)_{\varphi}\lrt 0,\]
where $f=(f_1,f_2)$ is defined by $f_1=0$ and $f_2((0,\alpha))=(\alpha, \varphi(\alpha))$ and $g=(g_1,g_2)$ is defined by $g_1=1_X$ and $g_2((\alpha,\beta))=\varphi(\alpha)-\beta.$\\

As in the case of quivers, these evaluation functors and their left adjoints can naturally be extended to the corresponding homotopy categories of bounded complexes, as well as to their bounded derived categories. Let us denote the extension functors by $k^i$ and their adjoints by $k^i_{\la}$, for $i=1,2$. Similar argument implies that the above short exact sequence can be extended to any bounded complex $X$ of $\Ga$-modules. That is, for any bounded complex $\CX=((X,Y)_{\varphi})_\bullet$ of $\Ga$-modules, there exists a short exact sequence of complexes of $\Ga$-modules
\[0 \lrt k^2_{\la}(M\otimes X_\bullet)\lrt k^1_{\la}(X_\bullet)\oplus k^2_{\la}(Y_\bullet) \lrt \CX \lrt 0,\]
where $X_\bullet$ (resp. $Y_\bullet$) is the complex of $R$-modules (resp. $S$-modules) arising from the modules in the first components (second components) of the complex $\CX$. This, in particular, induces an exact triangle
\[\xymatrix{k^2_{\la}(M\otimes X_\bullet) \ar[r] & k^1_{\la}(X_\bullet)\oplus k^2_{\la}(Y_\bullet) \ar[r] & \CX \ar@{~>}[r] & },\]
in $\D^b(Mod\Ga).$

\s {\sc Representation dimension of an artin algebra.}
Let $k$ be a commutative artin ring and $\La$ be an artin algebra. Our modules are always in $\mmod\La$. We also let $D( \ )$ denote the usual duality functor $\Hom_k( \ ,k)$.

Let $M$ be a $\La$-module. We denote by $\add M$ the full subcategory of $\mmod\La$ consisting of all direct summands of finite direct sums of copies of $M$.

A $\La$-module $M$ is called a generator if $\La \in \add M$. This is equivalent to say that $\add M$ contains all projective $\La$-modules. And also is equivalent to say that for any $\La$-module $X$, there exists an epimorphism $\varphi: M' \rt X$, with $M' \in \add M$.

Dually $M$ is called a cogenerator if $D(\La) \in \add M$. As above, this is equivalent to say that $\add M$ contains all injective $\La$-modules. And also is equivalent to say that for any $\La$-module $X$, there exists a monomorphism $\psi: X \rt M'$, with $M' \in \add M$.

A module which is both a generator and a cogenerator, will be called a generator-cogenerator. Auslander observed that the global dimension of the endomorphism ring of such modules have a vital role in studying the category $\mmod\La$.
He defined the representation dimension of $\La$ to be the minimum of the global dimension of the endomorphism rings of all generator-cogenerator $\La$-modules. As we mentioned in the introduction, it characterizes algebras of finite representation type, in the sense that $\La$ is of finite representation type if and only if its representation dimension is at most $2$. There is a long list of research papers in the literature studying representation dimension of different classes of artin algebras, and until Rouquier's example, all of them where at most $3$. For good references on this subject, see \cite{Au} and \cite{R}.

\s {\sc Dimension of triangulated categories.}
Let $\CT$ be a triangulated category. Let $\CX$, $\CX_1$ and $\CX_2$ be subcategories of $\CT$. Denote by $\lan\CX\ran$ the smallest full subcategory of $\CT$ which contains $\CX$ and is closed under taking finite coproducts, direct factors, and all shifts. Also let $\CX_1 \ast \CX_2$ denote the full subcategory of $\CT$ consisting of objects $X$ such that there exists a distinguished triangle
\[\xymatrix{X_1 \ar[r] & X \ar[r] & X_2 \ar@{~>}[r] & },\]
in $\CT$ with $X_i \in \CX_i$, for $i=1,2$. Finally, set $\CX_1 \diamond \CX2:=\lan\CX_1\ast \CX_2\ran$.

Inductively one defines ${\lan \CX\ran}_0 = 0$ and ${\lan \CX\ran}_n ={\lan \CX\ran}_{n-1}\diamond \lan \CX\ran$ for $n\geq 1$. By definition, the dimension of $\CT$, denoted by $\dim(\CT)$, is the minimal integer $d\geq 0$ such that there exists an object $M \in \CT$ with $\CT={\lan M\ran}_{d+1}$, or $\infty$ when there is no such an object $M$.

\section{Representation dimension of path algebras}
Let $\La$ be an artin algebra and \begin{equation*} T_2(\La) = \left(\begin{array}{ccc} \La & 0 \\\La & \La \end{array} \right)\end{equation*}denote the ring of $2\times 2$ lower triangular matrices over $\La$. It is shown in
\cite[Theorem 7.3]{FGR} that $\repdim T_2(\La) \leq \repdim \La + 2$. Note that $T_2(\La)$ is nothing but the path algebra of the quiver $A_2: v_1\rt v_2$. This result has been generalized to any quiver
\[v_1  \rt v_2 \cdots \rt v_n,\]
of type $A_n$, for arbitrary natural number $n$, by Xi \cite[Corollary 3.6]{X1}. Let $\La$ be an artin algebra over a perfect field $k$ and
\begin{equation*} T_n(\La) = \left(\begin{array}{cccc} \La & 0 & \cdots & 0 \\ \La & \La & \cdots & 0 \\ \vdots & \vdots & \ddots & \vdots \\ \La & \La & \cdots & \La \end{array} \right)\end{equation*}
be the $n\times n$ triangular matrix algebra with entries in $\La$. Then Xi proved that $\repdim T_n(A)\leq \repdim\La + 2$.

In this section we plan to extend this result to any finite acyclic quiver, which is not of type $A_n$. More precisely, we show that if $\La$ is an artin algebra and $\CQ$ is such a quiver, then $\repdim\La\CQ \leq \repdim\La+5.$ To prove this, we need some preliminary results. We begin with a lemma.

\begin{lemma}\label{pdimpath}
Let $R$ be a ring, $\CQ$ be a quiver and $\CX$ be a representation of $\CQ$ in $\mmod R$. Assume that, for any $v \in \CQ_0$, the $R$-module $e^v(\CX)=\CX_v$ has projective dimension at most $n$. Then $\pd_{\QR}\CX \leq n+1.$
\end{lemma}

\begin{proof}
By \ref{ExactTri}, for any $P \in \Prj(R)$ and any $v \in \CQ_0$, $e^v_{\la}(P)$ is a projective object in the category $\QR$. This easily implies that if $M \in \mmod R$ is of projective dimension at most $n$, then $e^v_{\la}(M)$, as an object in $\QR$, has projective dimension at most $n$. Now assume that $\CX$ is an arbitrary representation of $\CQ$ such that $\pd_R\CX_v\leq n$, for all $v \in \CQ_0$. By \ref{ExactTri}, $\CX$ can be embedded into an exact sequence \[0 \rt \bigoplus_{a \in \CQ_1} e^{t(a)}_\la(\CX_{s(a)}) \rt \bigoplus_{v \in \CQ_0} e^v_\la(\CX_v) \rt \CX \rt 0,\] of representations. Hence, the result follows from the fact that the first two left terms of this sequence has projective dimension at most $n$.
\end{proof}

As a direct consequence of the above lemma, we have the following corollary.

\begin{corollary}\label{gldimRQ}
Let $R$ be a ring and $\CQ$ be an arbitrary quiver. Then \[\gldim RQ \leq \gldim R+1.\]
\end{corollary}

\s \label{projmods} Let $\Ga=\left(\begin{array}{ccc} R & 0 \\ {}_SM_R & S \end{array} \right)$ be a formal triangular matrix ring. It follows from \cite[Corollart 1.6(c)]{FGR} that a $\Ga$-module $(X,Y)_{\varphi}$ is projective if and only if $X$ is a projective $R$-module, $\varphi:M\otimes_R X\lrt Y$ is an $S$-monomorphism and $\Coker(\varphi)$ is a projective $S$-module, see also \cite[Theorem 3.1]{HV}. This, in particular, implies that any projective $\Ga$-module can be written as a direct sum \[e^1_{\la}(P)\oplus e^2_{\la}(Q)=(P,M\otimes_RP)_1\oplus(0,Q)_0,\] where $P$ is a projective $R$-module and $Q$ is a projective $S$-module. We shall use this fact in the proof of the following lemma, see also \cite[Corollary 4.21]{FGR}.

\begin{lemma}\label{pdGAmodules}
Let $\Ga=\left(\begin{array}{ccc} R & 0 \\ {}_SM_R & S \end{array} \right)$ be a formal triangular matrix ring. Let $(X,Y)_{\varphi}$ be a $\Ga$-module for which ${}_RX$ and also ${}_SY$ have projective dimension at most $n$, for some natural number $n$. If ${}_SM$ is projective, then $\pd_{\Ga}((X,Y)_{\varphi})\leq n+1$.
\end{lemma}

\begin{proof}
By \ref{projmods}, a projective resolution of the $\Ga$-module $(X,Y)_{\varphi}$ can be written in the following form
\[\cdots \lrt (P_{n},M\otimes_RP_{n})_1\oplus(0,Q_{n})_0 \lrt \cdots \lrt (P_{0},M\otimes_RP_{0})_1\oplus(0,Q_{0})_0 \lrt (X,Y)_{\varphi} \lrt 0,\]
where $P_i\in \Prj(R)$ and $Q_i\in\Prj(S)$, for $i\geq 0$. Let $(K,L)_\psi$ be the $n$th syzygy of $(X,Y)_{\varphi}$. We show that $(K,L)_\psi$ has projective dimension at most one. The above resolution induces the following exact sequence
\[0 \lrt K \lrt P_{n-1} \lrt \cdots \lrt P_1 \lrt P_0 \lrt X \lrt 0,\]
of $R$-modules and an exact sequence
\[0 \lrt L \lrt (M\otimes_RP_{n-1})\oplus Q_{n-1} \lrt \cdots \lrt (M\otimes_RP_{1})\oplus Q_{1} \lrt (M\otimes_RP_{0})\oplus Q_{0} \lrt Y \lrt 0,\]
of $S$-modules. The projectivity of ${}_SM$ in conjunction with our assumption on the projective dimensions of $X$ and $Y$ imply that ${}_RK$ and ${}_SL$ are projective. Now consider the following short exact sequence of $\Ga$-modules, which exists in view of \ref{ForTriMatAlg},
\[0 \lrt (0,M\otimes_RK)_0 \lrt (K,M\otimes_RK)_1\oplus(0,L)_0 \lrt (K,L)_{\psi} \lrt 0.\]
Since $M \in \Prj(S)$, by using Theorem 3.1 of \cite{HV}, we deduce that the first two left terms are projective and so $\pd_{\Ga}((K,L)_{\psi}) \leq 1$. Therefore $\pd_{\Ga}((X,Y)_{\varphi})\leq n+1$.
\end{proof}

Following result can be considered as a wide generalization of the above lemma. For its proof see \cite[Corollary 4.21]{FGR}.

\begin{lemma}\label{FGRgldim}
Let $R$ and $S$ be rings and $M$ be an $S$-$R$-bimodule. Set $\Ga=\left(\begin{array}{ccc} R & 0 \\
M & S \end{array} \right).$ Then the following inequalities hold.
\[\max\{\gldim R, \ \gldim S, \ \pd_SM+1\} \leq \gldim\Ga \leq \max\{\gldim R+\pd_SM+1, \ \gldim S\}.\]
\end{lemma}

\vspace{0.5cm}

We also need the following two technical lemmas in the proof of the theorem.

\begin{lemma}\label{End}
Let $R$ be a ring, $\CQ$ be a finite quiver and $A$ be an $R$-module. Set $\Ga=\Hom_R(A,A)$. Then \[\End(\bigoplus_{v\in\CQ_0}e^v_{\la}(A)) \cong \Ga\CQ \cong \End(\bigoplus_{v\in\CQ_0}e^v_{\rho}(A)).\]
\end{lemma}

\begin{proof}
We prove the first isomorphism, the second one follows similarly. Let $\CQ$ has $n$ vertices. Then $\End(\bigoplus_{v\in\CQ_0}e^v_{\la}(A))$ can be considered as an $n$ by $n$ matrix whose $v\times w$ entry is  $E_{vw}:=\Hom_{\QR}(e^v_{\la}(A),e^w_{\la}(A))$. Using the adjoint duality $(e^v_{\la},e^v)$, we have
$E_{vw}\cong \Hom_{R}(A,(e^w_{\la}(A))_v) \cong \bigoplus_{\CQ(w,v)}\Ga$. The proof now can be completed easily.
\end{proof}

\begin{lemma}\label{vanish}
Let $R$ be a ring and $\CQ$ be a finite and acyclic quiver which is not of type $A_n$. Let $v\in S$ and $w\in \CQ_0\setminus S$, where $S$ denotes the set of all sinks. Then, for any $A\in \Mod R$, \[\Hom(e^v_{\rho}(A),e^w_{\la}(A))=0.\]
\end{lemma}

\begin{proof}
Assume contrary that there exists a non-zero morphism $\varphi \in \Hom(e^v_{\rho}(A),e^w_{\la}(A)).$ So there exists a vertex $v_1 \in \CQ_0$ such that ${\varphi}_{v_1}: {e^v_{\rho}(A)}_{v_1} \lrt {e^w_{\la}(A)}_{v_1}$ is a non-zero map. Without loss of generality we may assume that $v_1$ is a source, i.e., there is no arrow with target $v_1$. This follows directly from the definition of $e^v_{\rho}(A)$ in view of the fact that all maps in the representation $e^v_{\rho}(A)$ are epimorphisms. Let $\alpha$ denote the path from $v_1$ to $v$ and $\vartheta=\{v_1, v_2, \cdots, v_n=v\}$ be the set of all vertices existing on the path $\alpha$. Since $v\in S$ and $\CQ$ is not of type $A_n$, there should exists an arrow $a \rt b$ in $\CQ$ that does not lie on the path $\alpha$. So either $a \in \vartheta$ and $b \not\in \vartheta$ or conversely $b \in \vartheta$ and $a \not\in \vartheta$. Assume the first case happens. Consider the diagram
\[\xymatrix{{e^v_{\rho}(A)}_{a} \ar[r]^{{\varphi}_a} \ar[d]^{\xi_{ab}} & {e^w_{\la}(A)}_{a} \ar[d]^{\zeta_{ab}}\\
{e^v_{\rho}(A)}_{b} \ar[r]^{{\varphi}_b} & {e^w_{\la}(A)}_{b} }\]
Since ${\varphi}_{v_1}\neq 0$, it follows that ${\varphi}_{a}\neq 0$. Since $\zeta_{ab}$ is an injection, we get that $\zeta_{ab}{\varphi}_{a}\neq 0$. But it follows from our assumption that $\xi_{ab}=0$, contradicts with the commutativity of the diagram. So ${\varphi}_{v_1}$ should vanish in this case.

Now assume that $b \in \vartheta$ and $a \not\in \vartheta$. The non-vanishing of ${\varphi}_{v_1}$ implies that ${e^w_{\la}(A)}_{v_1}\neq 0$ and so there should exists an arrow from $w \rt v_{1}$. Since $v_1$ is a source, we should have $w=v_{1}$. Now it is easy, using the fact that $\CQ$ is not of type $A_n$, to show that there should exits an arrow $a'\rt a$ such that $a' \in \vartheta$. So the same argument as in the previous case applies to complete the proof.
\end{proof}

Now we are ready to state and prove our main theorem in this section. Let us preface it with a remark.

\begin{remark}\label{Steffen}
Let $k$ be a perfect field and $\La_1$ and $\La_2$ be two finite dimensional $k$-algebra. It was proved by Xi \cite[Theorem 3.5]{X1} that \[\repdim(\La_1\otimes\La_2)\leq \repdim\La_1 + \repdim\La_2.\] Therefore, since $\La\CQ \cong \La \otimes k\CQ$, we may deduce that \[\repdim\La\CQ \leq \repdim\La+\repdim k\CQ.\] This implies that $\repdim\La\CQ \leq \repdim\La+3,$ because $k\CQ$ is a hereditary algebra and for such algebras it is known that
$\repdim k\CQ\leq 3$, see e.g \cite[Corollary 5.2]{X1}. But in case we work with algebras over a commutative artin ring $R$, instead of a field, $R\CQ$ is not hereditary and so we do not have such bound on the representation dimension of $R\CQ$. In the following theorem, we get a bound on the representation dimension of such algebras. We thank Steffen Oppermann for pointing out this comment.
\end{remark}

\begin{theorem}\label{main1}
Let $\La$ be an artin algebra and $\CQ$ be a finite acyclic quiver which is not of type $A_n$. Then
\[\repdim\La\CQ \leq \repdim\La+5.\]
\end{theorem}

\begin{proof}
Assume that $\repdim\La=n$. By definition, there exists a generator-cogenerator $A \in \mmod\La$ such that $\gldim\End(A)=n.$ Let $S \subseteq \CQ_0$ be the set of all sinks in $\CQ$. In the sequel, just for simplicity, we shall write $V$ instead of $\CQ_0$ to denote the set of vertices of $\CQ$. Consider the following representations
\[\CX_1=\bigoplus_{v \in S}e^v_{\la}(A), \,\,\,\,\,\,\,\,\,\, \CX_2=\bigoplus_{v \in V\setminus S}e^v_{\la}(A)\oplus \bigoplus_{v\in S}e^v_{\rho}(A),\,\,\,\,\,\,\,\,\,\, \CX_3=\bigoplus_{v \in V\setminus S}e^v_{\rho}(A),\]
and set \[\overline{\CX}=\CX_1 \bigoplus \CX_2 \bigoplus \CX_3.\]
The structure of projective and injective representations of $\CQ$ described in \ref{adj}, implies easily that
$\overline{\CX}$ is a generator-cogenerator for $\La\CQ$. We show that $\gldim\End(\overline{\CX})\leq n+5$.

An easy computation of morphisms in $\QLA$ show that the following four Hom-sets vanish.
\[\Hom_{\La\CQ}(\CX_1,\CX_3)=\Hom_{\La\CQ}(\CX_2,\CX_1)=\Hom_{\La\CQ}(\CX_3,\CX_1)=\Hom_{\La\CQ}(\CX_3,\CX_2)=0.\]
Thus
\[\End(\overline A)=\End(\CX_1\oplus \CX_2 \oplus \CX_3)\cong  \left( \begin{array}{ccc}
{\End(\CX_1)} & 0 & 0 \\ \Hom(\CX_1,\CX_2) & \End(\CX_2) & 0 \\ 0 & \Hom(\CX_2,\CX_3) & \End(\CX_3) \end{array} \right).\]
We will compute the remaining Hom sets using the adjoint properties and also the structure of functors $e^v, e^v_{\la}$ and $e^v_{\rho}$, in conjunction with the above lemma. Set $\Ga=\Hom_{\La}(A,A)$. It follows from Lemma \ref{End} that $\End(\CX_1)\cong\prod_{|S|}\Ga$ and $\End(\CX_3)\cong\Ga\CQ',$ where $\CQ'$ is the subquiver of $\CQ$ with $V(\CQ')=V(\CQ)\setminus S$. To compute $\End(\CX_2)$, we should compute four Hom-sets. Lemma \ref{End} implies that the first Hom set, which is
\[\Hom(\oplus_{v \in V\setminus S}e^v_{\lambda}(A),\oplus_{v \in V\setminus S}e^v_{\lambda}(A))\] is isomorphic to the path algebra $\Ga\CQ'$, where as above $\CQ'$ is the full subquiver of $Q$ with $V(\CQ')=V(\CQ)\setminus S$. Lemma \ref{vanish} implies that the second one, that is \[\Hom(\oplus_{v\in S}e^v_{\rho}(A),\oplus_{v \in V\setminus S}e^v_{\lambda}(A))\] should be zero. Let us temporarily set \[M:=\Hom(\oplus_{v \in V\setminus S}e^v_{\lambda}(A),\oplus_{v\in S}e^v_{\rho}(A)).\]
Finally, we use Lemma \ref{End} once more, to obtain the forth Hom set
\[\Hom(\oplus_{v\in S}e^v_{\rho}(A),\oplus_{v\in S}e^v_{\rho}(A))\cong \prod_{|S|}\Ga.\]
Therefore
\[\End(\CX_2)= \left(\begin{array}{ccc}\Ga\CQ' & 0 \\ M & \prod_{|S|}\Ga\end{array} \right).\]

Now, we plan to discuss their global dimensions. It is clear that $\gldim\End(\CX_1)=\gldim\Ga=n$. As for $\End{\CX_2}$, first note that $M$ as a left $\prod_{|S|}\Ga$-module is projective. This follows from the fact that $M$ is isomorphic to some coproduct of copies of $\Ga$. On the other hand, it follows from Lemma \ref{pdimpath}, that $\gldim(\Ga\CQ')$, is less than or equal to $\gldim\Ga+1=n+1$. Therefore, in view of Lemma \ref{FGRgldim},
\[\gldim\End(\CX_2)\leq \max\{n+1+0+1, \ n\}=n+2.\]
Since $\End(\CX_3)\cong\Ga\CQ',$ Lemma \ref{pdimpath} implies that $\gldim\End(\CX_3)\leq \gldim\Ga+1=n+1.$

To continue, set
\[\Sigma=  \left(\begin{array}{ccc}{\End(\CX_1)}& 0 \\\Hom(\CX_1,\CX_2) & \End(\CX_2)\end{array} \right).\]
To be able to apply Lemma \ref{FGRgldim} to get the global dimension of $\Sigma$, we have to compute the projective dimension of $\Hom(\CX_1,\CX_2)$ as a left $\End(\CX_2)$-module. By \ref{ForTriMatAlg}, $\Hom(\CX_1,\CX_2)$ as a module over $\End(\CX_2)$ can be represented as a triple $(X,Y)_{\varphi}$. $X$ as a representation over $\CQ'$ is projective in each vertex so by Lemma \ref{pdimpath}, $\pd_{\Ga Q'}X\leq 1$. Moreover, $Y$ is projective over $\prod_{|S|}\Ga$. Hence by Lemma \ref{pdGAmodules} the projective dimension of $(X,Y)_{\varphi}$ over $\End(\CX_2)$ is at most $2$.
Therefore Lemma \ref{FGRgldim} implies that
\[\begin{array}{lll} \gldim\Sigma & \leq \max\{\gldim\End(\CX_1)+\pd_{\End(\CX_2)}\Hom(\CX_1,\CX_2)+1, \ \gldim\End(\CX_2)\}\\ & \leq \max\{n+2+1, \ n+2\}\\ & = n+3.\end{array}\]

Finally, to compute the global dimension of $\End(\overline{\CX})$ first we note that $\Hom(\CX_2,\CX_3)$, as a representation over $\End(\CX_3)\cong\Ga\CQ'$, has projective dimension at most one. This follows from Lemma \ref{pdimpath}, in view of the fact that $\Hom(\CX_2,\CX_3)$ as a representation over $\Ga\CQ'$ is projective at each vertex. Therefore,
\[\begin{array}{lll} \gldim\End(\overline{\CX}) & \leq \max\{\gldim\Sigma+\pd_{\End(\CX_3)}\Hom(\CX_2,\CX_3)+1, \ \gldim\End(\CX_3)\}\\ &
\leq \max\{n+3+1+1, \ n+1\} \\ & = n+5.\end{array}\]
The proof is now complete.
\end{proof}

\begin{remark}
Assume that $\CQ$ is a quiver of type $D_4$. If we define $\CX_1$ and $\CX_2$ as above and consider the Hom set $\Hom(\CX_1,\CX_2)$, it is easy to see that it is not projective as an $\End(\CX_2)$-module.
\end{remark}

\section{Derived dimension of path algebras}
In this section, we compare the dimension of the bounded derived category of the path algebra $RQ$, where $R$ is a coherent ring and $Q$ is a finite and acyclic quiver, with the dimension of the bounded derived category of $R$-modules.

The following lemma is known, see \cite[Lemma 3.4]{Rq2}.

\begin{lemma}\label{denceimage}
Let $F : \CT \rightarrow \CT'$ be a triangulated functor such that any object in $\CT'$ is isomorphic to an object in the image of $F$. Then $\dim\CT' \leq \dim\CT$. In fact, if $\CT={\lan\CX\ran}_d$, then $\CT'={\lan F(\CX)\ran}_d$.
\end{lemma}

Recall that a (left) coherent ring is a ring in which every finitely generated left ideal is finitely presented. Being coherent is equivalent to say that the category $\mmod R$ of finitely presented right $R$-modules is abelian.

\begin{theorem}
Let $R$ be a coherent ring and $\CQ$ be a finite acyclic quiver. Then
\[\dim\D^b(\mmod R) \leq \dim\D^b(\mmod R\CQ) \leq 2\dim\D^b(\mmod R)+1.\]
\end{theorem}

\begin{proof}
Let $v \in \CQ_0$ be arbitrary and consider the functor $k^v:\D^b(\mmod R\CQ) \rt \D^b(\mmod R)$. It is clear that this functor satisfies the statement of Lemma \ref{denceimage} and so we may deduce that $\dim\D^b(\mmod R) \leq \dim\D^b(\mmod R\CQ)$. If $\dim\D^b(\mmod R)=\infty$, the proof is complete. So assume that $\dim\D^b(\mmod R)=n$  is finite. Let $X \in \D^b(\mmod R)$ be such that $\D^b(\mmod R)={\lan X\ran}_{n+1}$. We show that $\D^b(\mmod R \CQ)={\lan \bigoplus_{v \in \CQ_0} k^v_{\la}(X)\ran}_{2n+2}$. To this end, first we claim that for $0 \leq i\leq n+1$ and $v \in \CQ_0$,
\[k^v_{\la}({\lan X\ran}_i)\subseteq \lan\bigoplus_{v \in \CQ_0}k^v_{\la}(X)\ran_i\]
For $i=0,1$ it is clear by definition. So assume inductively that $i>2$ and the result has been proved for smaller value than $i$. Let $C \in {\lan X\ran}_i$. By definition $C$ is a direct summand of an object $M$, for which there is a triangle
\[\xymatrix{I \ar[r] & M \ar[r] & J \ar@{~>}[r] & },\]
with $I \in {\lan X\ran}_{i-1}$ and $J \in \lan X\ran$. By applying the functor $k^v_{\la}$ on the above triangle, we get the exact triangle
\[\xymatrix{k^v_{\la}(I) \ar[r] & k^v_{\la}(M) \ar[r] & k^v_{\la}(J) \ar@{~>}[r] & },\]
in $\D^b(\mmod R)$. But by induction hypotheses $k^v_{\la}(I) \in {\lan \bigoplus_{v \in \CQ_0} k^v_{\la}(X)\ran}_{i-1}$ and by definition $k^v_{\la}(J) \in {\lan\bigoplus_{v \in \CQ_0}k^v_{\la}(X)\ran}$. Therefore, since $k^v_{\la}(C)$ is a direct summand of $k^v_{\la}(M)$, it follows that $k^v_{\la}(C) \in {\lan\bigoplus_{v \in \CQ_0}k^v_{\la}(X)\ran}_i$.
In particular, for $i=n+1$,
\[k^v_{\la}(\D^b(\mmod R))=k^v_{\la}({\lan X\ran}_{n+1}) \subseteq
\lan\bigoplus_{v \in \CQ_0}k^v_{\la}(X)\ran_{n+1}.\]
The above discussion implies that for any $X \in \D^b(\mmod R)$ and any $v \in \CQ_0$, $k^v_{\la}(X) \in {\lan\bigoplus_{v \in \CQ_0}k^v_{\la}(X)\ran}_{n+1}$.

On the other hand, by \ref{ExactTri}, any complex $ \CX \in \D^b({\mmod R \CQ})$ fits in to triangle
\[\xymatrix{\bigoplus_v k^v_{\la}(\CX_v) \ar[r] & \CX \ar[r] & \sum \bigoplus_a k^{t(a)}_{\la}(\CX_{s(a)}) \ar@{~>}[r] & }.\]
Therefore $\D^b({\mmod R\CQ})=\lan{\lan\bigoplus_{v \in \CQ_0}k^v_{\la}(X)\ran}_{n+1}\ast{\lan\bigoplus_{v \in \CQ_0}k^v_{\la}(X)\ran}_{n+1}\ran$. The associativity of the operation $\ast$, implies that
$\D^b({\mmod R\CQ})=\lan{\lan\bigoplus_{v \in \CQ_0}k^v_{\la}(X)\ran}_{2n+1}\ast \lan X\ran\ran$. This finishes the proof.
\end{proof}

Let $\La$ be a finite dimensional algebra over an algebraically closed field $k$. It is proved in
\cite[Theorem 12.20]{B1} that $\derdim\La=0$ if and only if $\La$ is an iterated tilted algebra of Dynkin type, see also \cite{CYZ}. Using this fact, the following corollary follows immediately from the proof of the above theorem.

\begin{corollary}
Let $k$ be an algebraically closed field and $\CQ$ be a (finite and acyclic) quiver.
If $\CQ$ is not a Dynkin quiver, then $\dim\D^b(\mmod k\CQ)=1$.
\end{corollary}

\begin{proof}
Clearly $\dim\D^b(\mmod k)=0$. So it follows from the above theorem that $\dim\D^b(\mmod k\CQ)\leq 1$. If $\CQ$ is not Dynkin, by \cite{CYZ}, $\dim\D^b(\mmod k\CQ)\neq 0$ and hence it has to be one. The proof is hence complete.
\end{proof}

\begin{remark}
Let $k$ be an algebraically closed field and $\CQ$ be a Dynkin quiver. Then $\derdim k\CQ=0$ while $\gldim k\CQ=1$. This provide an example of an algebra of finite global dimension whose global dimension is strictly larger than the derived dimension.
\end{remark}

A similar technique, as we used in the proof of the above theorem, can be applied for the derived dimension of a formal triangular matrix ring.

Let $\Ga= \left( \begin{array}{ccc} R & 0 \\ M & S \end{array} \right)$ be a formal triangular matrix ring, where $R$ and $S$ are two rings and ${}_SM_R$ is an $S$-$R$-bimodule. By \cite[Corollary 2.3]{FGR}, we have a complete characterization of $\Ga$ when it is (left) coherent. In particular, it is shown that if $\Ga$ is coherent, then so are $R$ and $S$. So we may record the following result.

\begin{theorem}
Let $\Ga= \left( \begin{array}{ccc} R & 0 \\ M & S \end{array} \right)$ be a coherent formal triangular matrix ring. Set $n=\max\{\dim\D^b(\mmod R), \dim\D^b(\mmod S)\}$. Then
\[n \leq \dim\D^b(\mmod \Ga) \leq 2n+1.\]
\end{theorem}

\begin{proof}
It is clear that two functors
\[k^1: \D^b(\mmod \Ga) \rt \D^b(\mmod R) \ \ \ \ {\rm and} \ \ \ \ k^2: \D^b(\mmod \Ga) \rt \D^b(\mmod S) \]
are dense and so the first inequality follow from Lemma \ref{denceimage}. If
\[\max\{\dim\D^b(\mmod R), \dim\D^b(\mmod S)\}=\infty,\] we are done. So assume that it is finite. Let $X_1 \in \D^b(\mmod R)$ and $X_2 \in \D^b(\mmod S)$ be such that $\D^b(\mmod R)={\lan X_1\ran}_{n+1}$ and $\D^b(\mmod S)={\lan X_2\ran}_{n+1}$. Now one can apply the same method as in the proof of the above theorem, to show that $\D^b(\mmod \Ga)={\lan k^1_{\la}(X_1)\oplus k^2_{\la}(X_2)\ran}_{2n+2}$.
Hence the proof is complete.
\end{proof}

\section*{Acknowledgments}
We would like to thank Apostolos Beligiannis for reading carefully the first draft of this paper and for his useful comments and hints that improved our exposition. We also would like to thank Steffen Oppermann and Idun Reiten for their useful comments. The authors thank the Center of Excellence for Mathematics (University of Isfahan).

\end{document}